\documentclass[a4paper,12pt]{article}

\usepackage[textwidth=125mm, textheight=195mm]{geometry}
\usepackage[english]{babel}
\usepackage{graphicx}
\usepackage{amsmath}
\usepackage{amsfonts}
\usepackage{amsthm}
\usepackage{epstopdf}
\usepackage{color}

\begin{document}

\newcommand{\wk}{\mbox{$\,<$\hspace{-5pt}\footnotesize )$\,$}}

\numberwithin{equation}{section}
\newtheorem{teo}{Theorem}
\newtheorem{lemma}{Lemma}

\newtheorem{coro}{Corollary}
\newtheorem{prop}{Proposition}
\theoremstyle{definition}
\newtheorem{definition}{Definition}
\theoremstyle{remark}
\newtheorem{remark}{Remark}

\newtheorem{scho}{Scholium}
\newtheorem{open}{Question}
\newtheorem{example}{Example}
\numberwithin{example}{section}
\numberwithin{lemma}{section}
\numberwithin{prop}{section}
\numberwithin{teo}{section}
\numberwithin{definition}{section}
\numberwithin{coro}{section}
\numberwithin{figure}{section}
\numberwithin{remark}{section}
\numberwithin{scho}{section}

\bibliographystyle{abbrv}

\title{Differential geometry of immersed surfaces in three-dimensional normed spaces}
\date{}

\author{Vitor Balestro\footnote{Corresponding author}  \\ CEFET/RJ Campus Nova Friburgo \\ 28635000 Nova Friburgo \\ Brazil \\ vitorbalestro@id.uff.br \and Horst Martini \\ Fakult\"{a}t f\"{u}r Mathematik \\ Technische Universit\"{a}t Chemnitz \\ 09107 Chemnitz\\ Germany \\ martini@mathematik.tu-chemnitz.de \and  Ralph Teixeira \\ Instituto de Matem\'{a}tica e Estat\'{i}stica  \\ Universidade Federal Fluminense \\24210201 Niter\'{o}i\\ Brazil \\ ralph@mat.uff.br}

\maketitle

\begin{abstract}
In this paper we study curvature types of immersed surfaces in three-dimensional (normed or) Minkowski spaces. By endowing the surface with a normal vector field, which is a transversal vector field given by the ambient Birkhoff orthogonality, we get an analogue of the Gauss map. Then we can define concepts of principal, Gaussian, and mean curvatures in terms of the eigenvalues of the differential of this map. Considering planar sections containing the normal field, we also define normal curvatures at each point of the surface, and with respect to each tangent direction. We investigate the relations between these curvature types. Further on we prove that, under an additional hypothesis, a compact, connected surface without boundary whose Minkowski Gaussian curvature is constant must be a Minkowski sphere. 

\end{abstract}


\bigskip


\section{Introduction}

The geometry of normed spaces, also called \emph{Minkowski geometry} (see \cite{thompson}), is a research field which presents a lot of interesting directions to explore. Recently, we are concerned with giving a still missing systematic study of the differential geometry referring to these spaces. Although this topic was already partially studied throughout the past decades by researchers as Busemann (cf. \cite{Bus3}) and Petty (cf. \cite{Pet}), the existing works on it are, in a certain way, dispersed through the literature. The first step towards a systematic treatment is given in \cite{Ba-Ma-Sho}, where various concepts of curvature for regular curves in normed planes are studied. The present work aims to present an approach to the geometry of surfaces in three-dimensional Minkowski spaces, and this effort can be seen from the perspectives of different areas, namely Minkowski geometry, classical differential geometry, and Finsler geometry. In view of Minkowski geometry, we are extending some concepts from Euclidean geometry and investigate their behavior, asking what properties the general case has, and by which properties the inner product subcase can be characterized. Our construction is a particular equiaffine immersion (as it is defined and investigated in \cite{nomizu}), and we may study whether this particular case has a special behavior. From the viewpoint of Finsler geometry, we are studying curvatures of Finsler manifolds whose geometry is induced by the Minkowski geometry of an ambient space. The geometry in the tangent spaces is therefore defined by the parallel central planar sections of a certain convex body (namely, the unit sphere of the considered norm) that they determine.   \\

Let us concretely explain the initial idea behind our constructions. The approach to affine differential geometry given in the book \cite{nomizu} regards hypersurfaces immersed in an affine space $\mathbb{R}^{n+1}$ endowed with a \emph{transversal vector field} ``playing the role" of a normal vector field. Our idea in this paper is to endow a surface in a three-dimensional Minkowski space with the transversal vector field obtained via the Birkhoff orthogonality associated to the norm. As far as the authors know, this construction was originally studied by Biberstein \cite{biberstein}; but this approach was different to the one that we present here. As we will see, such transversal fields give birth to immersions which are \emph{equiaffine}, and this allows us to extend some concepts from classical differential geometry to normed spaces. As it is often the case when dealing with non-Euclidean geometries and, in particular, general Minkowski spaces, one is mainly confronted with two situations: a large variety of concepts obtained as very natural extensions yields also really different analogues of classical concepts, or it yields Minkowskian analogues which remain very similar to their Euclidean subcases. In the present paper, both variants of extensions will occur. \\

The paper is organized as follows. Sections \ref{basics} and \ref{affine} are devoted to background theory and notation from Minkowski geometry and affine differential geometry, respectively. The reader not familiar with some of these topics will find standard references. In particular, by endowing a regular curve with the normal vector field determined by Birkhoff orthogonality, we re-obtain the concept of \emph{circular curvature}, which is one of the curvature types studied for the planar situation in \cite{Ba-Ma-Sho}. This is an important concept for us, since it is central in the definition of an analogue to the \emph{normal curvature} of a surface. \\

The possibly most important part of this paper, Section \ref{birkgaussmap}, is devoted to the study of  an analogue to the Gauss map of a surface. We define the principal curvatures and principal directions to be the eigenvalues and eigenvectors of the differential of this map, respectively. Also, we extend the notions of curvature lines, asymptotic directions, and umbilic points. In particular it is proven that a surface whose points are all umbilic must be a Minkowski sphere. In Section \ref{normal} an analogue of normal curvature is obtained by considering the plane curvatures of curves obtained as plane sections of the surface. The relations between the normal curvature and the principal curvatures are also described. Up to assuming a further hypothesis on the (Minkowski) curvature lines of the surface, we prove in Section \ref{rigid} that, similarly to the Euclidean subcase, if the Minkowski Gaussian curvature of a surface is constant, then this surface must be a Minkowski sphere. We also prove that the same holds if the mean curvature is constant, subject to the hypothesis that the Minkowski Gaussian curvature is positive.\\

This is the first of three papers devoted to the study of differential geometry of surfaces immersed in three-dimensional spaces endowed with a norm from several viewpoints. Hence this paper aims to build the ``main core" of the theory, and this is mostly done in Sections \ref{birkgaussmap} and \ref{normal}. Also, the reader will notice that the investigations made in this paper are driven by the inspiration from classical differential geometry, since most of our results are analogues to the ones of this area subject. In \cite{diffgeom2} we adopt the viewpoint of affine differential geometry, which is also very natural, since our Gauss map is constructed using a central idea of this field. The third paper \cite{diffgeom3} deals with further topics in the theory, such as minimal surfaces and metric issues. \\

\section{Notation and basic concepts}\label{basics}

A \emph{normed} (=\emph{Minkowski}) \emph{space} is a finite dimensional vector space endowed with a usual norm. (Thus, to avoid confusion, we note  that we are not concerned with the so-called Minkowskian space-time geometry.) For the sake of simplicity, we assume here that the vector space is $\mathbb{R}^{n+1}$, and we will denote the norm by $||\cdot||$. The \emph{unit ball} and the \emph{unit sphere} of a normed space are, respectively, the sets $B = \{v \in \mathbb{R}^{n+1}:||v||\leq 1\}$ and $\partial B=\{v \in \mathbb{R}^{n+1}:||v|| = 1\}$, respectively. Thus, $B$ is a \emph{centered convex body}, i.e., a compact, convex set with non-empty interior centered at the origin. Since we are dealing with aspects of differential geometry, throughout the text we will assume that the norm is \emph{strictly convex} (meaning that the triangle inequality is strict for linearly independent vectors or, equivalently, that the unit sphere does not contain straight line segments) and \emph{smooth}, in the sense that the unit sphere is locally the graph of a $C^{\infty}$ map $f:U\subseteq\mathbb{R}^n\rightarrow\mathbb{R}$ (the $C^{\infty}$ hypothesis can, however, often be relaxed to $C^2$ or to the smallest regularity class such that all involved derivatives make sense). \\

The homothets of the unit ball are called \emph{Minkowski balls}, and their boundaries are the \emph{Minkowski spheres}. Given a non-zero vector $v \in \mathbb{R}^{n+1}$ and a hyperplane $H \subseteq \mathbb{R}^{n+1}$, we say that $v$ is \emph{Birkhoff orthogonal} (or simply \emph{orthogonal}) to $H$ if the unit ball is supported by $H$ at $v/||v||$ (see Figure \ref{birkhoff}). We denote this relation by $v \dashv_B H$. From the hypothesis of smoothness and strict convexity it follows that Birkhoff orthogonality is unique, both at left and at right (as a relation between directions and planes). \\

\begin{figure}[h]
\centering
\includegraphics{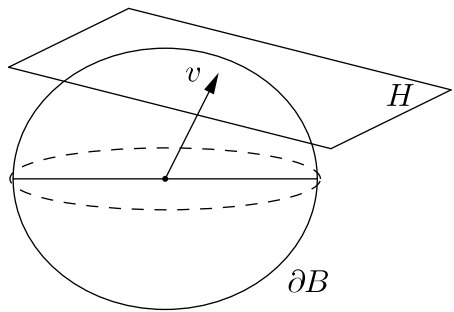}
\caption{$v \dashv_B H$.}
\label{birkhoff}
\end{figure}

Notice that Birkhoff orthogonality can be naturally extended to a relation between vectors: given two non-zero vectors $v, w \in \mathbb{R}^{n+1}$, we say that $v$ is Birkhoff orthogonal to $w$ (denoted by $v \dashv_B w$) if the hyperplane $H$ to which $v$ is Birkhoff orthogonal contains $w$. If the space is two-dimensional, then Birkhoff orthogonality is essentially a (homogeneous) relation between vectors, since in this case the hyperplanes are lines (see \cite{alonso} for more about orthogonality relations in normed spaces). In spaces of dimensions at least 3, Birkhoff orthogonality is a symmetric relation if and only if the norm is Euclidean. In the two-dimensional case, Birkhoff orthogonality is symmetric if and only if the unit circle is a \emph{Radon curve}, yielding the notion of \emph{Radon plane}. Properties and constructions of Radon curves are presented in \cite{martiniantinorms} and \cite{radonbalestro}. Standard references with respect to Minkowski geometry as a whole are \cite{thompson}, \cite{martini2}, and \cite{martini1}. Regarding differential geometry in Minkowski spaces we refer the reader to \cite{keti} and to the papers \cite{Bus3} and \cite{Pet} already mentioned in the introduction. The paper \cite{keti} deals with global theorems for curves in Minkowski spaces (such as the Fenchel and the Fary-Milnor theorem). \\

Now we will introduce a concept of curvature for plane curves that we will use later. For more information on curvature concepts of curves in two-dimensional Minkowski spaces (i.e., in normed planes) we refer to \cite{Ba-Ma-Sho}. An immersed hypersurface in a normed plane is a regular curve $\gamma:[0,c]\rightarrow (\mathbb{R}^2,||\cdot||)$ which, for our purpose, we may assume to be parametrized by arc-length. Up to choice of orientation, the unit transversal vector field to $\gamma$ given by Birkhoff orthogonality is the vector field $\eta:[0,c]\rightarrow \partial B$ such that $\eta(s) \dashv_B \gamma(s)$ and $[\eta(s),\gamma'(s)] >0$ for each $s \in [0,c]$. Since $\eta(s) \in \partial B$, there exists a function $t(s):[0,c]\rightarrow[0,l(\partial B)]$ such that $\eta(s) = \varphi(t(s))$ for each $s \in [0,c]$. Consequently, we have that
\begin{align*} \frac{d\varphi}{dt}(t(s)) = \gamma'(s), \ \ s \in [0,c].
\end{align*}
In \cite{Ba-Ma-Sho} the number $k(s):=t'(s)$ is called the \emph{circular curvature} of $\gamma$ at $p = \gamma(s)$. It is a natural extension of the usual curvature in the Euclidean plane and, when it does not vanish, it can also be regarded as the inverse of the radius of the osculating (Minkowski) circle of $\gamma$ at $p$.

\section{Affine immersions and transversal vector fields} \label{affine}

As explained in the introducion, our motivation to approach differential geometry in normed spaces is inspired mainly by affine differential geometry, and for this theory we refer to the book \cite{nomizu}. In this short section we briefly explain the basic ideas and concepts; for proofs the reader can consult the mentioned book. \\

Assume that $\mathbb{R}^{n+1}$ is the usual $(n+1)$-dimensional Euclidean space endowed with the standard connection $D:C^{\infty}(T\mathbb{R}^{n+1})\times C^{\infty}(T\mathbb{R}^{n+1})\rightarrow C^{\infty}(T\mathbb{R}^{n+1})$, and let $M$ be a $n$-dimensional manifold together with an immersion $f:M\rightarrow\mathbb{R}^{n+1}$. We call $f$ a \emph{hypersurface immersion} of $\mathbb{R}^{n+1}$. In classical differential geometry, one usually uses, for $x \in M$, the decomposition
\begin{align*} T_{f(x)}\mathbb{R}^{n+1} = f_{*}(T_xM) + \mathrm{span}\{\eta_x\},
\end{align*}
where $f_{*}$ denotes the usual push-forward differential map, and $\eta_x$ is a unit normal vector to $f_{*}(T_xM)$. By doing so, one can obtain the Levi-Civita connection and the second fundamental form associated to the metric induced on $M$ by the Euclidean metric in $\mathbb{R}^{n+1}$. \\

However, the point here is that in order to decompose $T_{f(x)}\mathbb{R}^{n+1}$, one does not need to regard the \emph{normal vector} to $f_{*}(T_xM)$, but any \emph{transversal vector} to it. Formally, given a point $x_0 \in M$, we consider a smooth transversal vector field $x \mapsto \xi_x$ defined in a neighborhood $U$ of $x_0$, meaning that $\xi_x \notin f_{*}(T_xM)$. We write then
\begin{align}\label{decomposition} T_{f(x)}\mathbb{R}^{n+1} = f_{*}(T_xM) + \mathrm{span}\{\xi_x\}.
\end{align}
The choice of such a transversal field induces implicitly a connection $\nabla:C^{\infty}(TU)\times C^{\infty}(TU)\rightarrow C^{\infty}(TU)$ by
\begin{align}\label{gauss} D_Xf_{*}(Y) = f_{*}(\nabla_XY) + h(X,Y)\xi.
\end{align}
It is easy to see that $\nabla$ is a torsion-free connection, meaning that $\nabla_XY-\nabla_YX-[X,Y]$ vanishes everywhere, and $h$ is a symmetric bilinear form in $T_xM$, for each $x \in U$. The connection $\nabla$ is called the \emph{induced connection}. Formula (\ref{gauss}) is known as the \emph{formula of Gauss}. The map $h$ is called the \emph{affine fundamental form}, and its rank is called the \emph{rank of the hypersurface}. This number does not depend on the transversal field considered in the hypersurface (this is proved in \cite{nomizu}). \\

The covariant derivative of $\xi$ can be decomposed in view of (\ref{decomposition}) as well. By doing so, we obtain the \emph{Weingarten formula}
\begin{align}\label{weingarten} D_X\xi = -f_{*}(SX) + \tau(X)\xi,
\end{align}
where $S:T_xM\rightarrow T_xM$ is a linear map called the \emph{shape operator}, and $\tau:T_xM\rightarrow\mathbb{R}$ is a $1$-form. As we will show later, this $1$-form will be of no importance for us, since every considered transversal field will be shown to have tangential derivative. Let $\mathrm{det}:\mathbb{R}^{n+1}\rightarrow\mathbb{R}$ denote the usual determinant in $\mathbb{R}^{n+1}$, and define a volume form on $M$ by
\begin{align*} \omega(X_1,...,X_n) = \mathrm{det}[X_1,...,X_n,\xi],
\end{align*}
where we are omitting the push-forward $f_{*}$ from the notation, for the sake of convenience. We call $\theta$ the \emph{induced volume element}. A volume form $\theta$ is said to be \emph{parallel} with respect to a connection $\nabla$ if $\nabla\theta = 0$. For the induced volume element $\omega$, we have the following

\begin{prop} We have $\nabla_X\omega = \tau(X)\omega$ for each $X \in T_xM$. Consequently, $\omega$ is parallel if and only if $D_X\xi$ is always tangential. In this case, we say that $(\nabla,\omega)$ is an \emph{equiaffine structure} on $M$.
\end{prop}
\begin{proof} See \cite[Proposition 1.4]{nomizu}.

\end{proof}

If $\tau = 0$, we say that $f$ is an \emph{equiaffine immersion}. The affine fundamental form $h$ also induces a volume form $\omega_h$ on $M$, where $\omega_h(X_1,...,X_n)$ is the square root of the absolute value of the determinant of the matrix $[h_{ij}]$, whose entries are the numbers $h(X_i,X_j)$, for $i,j = 1,...,n$. An equiaffine immersion for which $\omega = \omega_h$ is said to be a \emph{Blaschke immersion}, and in this case the transversal field $\xi$ is called the \emph{affine normal field}.

\section{The Birkhoff-Gauss map} \label{birkgaussmap}

Let $(\mathbb{R}^3,||\cdot||)$ be a normed space whose norm is smooth and strictly convex, and let $f:M\rightarrow(\mathbb{R}^3,||\cdot||)$ be a surface immersion. In what follows, we naturally identify the manifold $M$ with its image $f(M)$ and, consequently, the tangent space $T_pM$ with $f_*(T_pM) \subseteq T_{f(p)}\mathbb{R}^{3}$; therefore we will omit the push-forward map $f_*$ from our notation.  For each $p \in M$, let $\eta(p) \in \partial B$ be a unit vector such that $\eta(p) \dashv_B T_pM$. The choice of such a transversal vector gives a local smooth unit vector field that, in our context, clearly plays the role of the usual normal vector field. We will call this vector field the \emph{Birkhoff normal vector field} of $M$. Of course, this local field can be regarded as global if and only if $M$ is orientable. As in the Euclidean subcase, the Birkhoff normal vector field can be seen as a map $\eta:U\subseteq M\rightarrow \partial B$, which we will call the \emph{Birkhoff-Gauss map}.

\begin{lemma} A surface immersion $f:M\rightarrow(\mathbb{R}^3,||\cdot||)$ endowed with the Birkhoff normal vector field is an equiaffine immersion.
\end{lemma}
\begin{proof} Let $p \in M$, and fix a vector $X \in T_pM$. Let $\gamma:(-\varepsilon,\varepsilon)\rightarrow M$ be any smooth curve such that $\gamma(0) = p$ and $\gamma'(0) = X$. Therefore, we have
\begin{align*} D_X\eta|_p = \frac{d}{dt}\eta\circ\gamma(t)|_{t=0}.
\end{align*}
On the other hand, $\eta\circ\gamma$ is a curve on $\partial B$, and hence the right hand term must be a vector in $T_{\eta(p)}\partial B$. Since $\eta(p) \dashv_B T_{\eta(p)}\partial B$, it follows that we have a natural identification $T_{\eta(p)}\partial B \simeq T_pM$. Then we have indeed $D_X\eta|_p \in T_pM$, and the covariant derivative of the transversal vector field is always tangential. This concludes the proof.

\end{proof}

In other words, the differential map $d\eta_p:T_pM\rightarrow T_{\eta(p)}\partial B$ can be regarded as a (linear) map of $T_pM$ onto itself. Notice also that $D_X\eta|_p = d\eta_p(X)$. In classical differential geometry, the differential of the Gauss map is used to define the curvatures of a surface at a point (see \cite{manfredo}). We will develop a similar theory for our context. We start by proving the key fact that each tangent space $T_pM$ admits a basis of eigenvectors of $d\eta_p$. We separate our approach into cases depending on the structure of the affine fundamental form. First, we assume that $h$ has rank $2$ and is definite. As we will see, for all the other cases we need an auxiliary Euclidean structure.

\begin{prop}\label{rank2} If the rank of the affine fundamental form $h$ equals $2$ at $p \in M$, and $h$ is definite, then the tangent space $T_pM$ admits a basis of eigenvectors of the differential map $d\eta_p:T_pM\rightarrow T_pM$.
\end{prop}
\begin{proof} Recall that the affine fundamental form $h$ is defined intrinsically by the Gauss formula
\begin{align*} D_XY = \nabla_XY + h(X,Y)\eta,
\end{align*}
where $X,Y \in T_pM$ and $\nabla$ is the induced connection. If $h$ has rank $2$, then we may choose vectors $X,Y \in T_pM$ such that $h(X,X) = h(Y,Y)  \neq 0$ and $h(X,Y) = 0$. Taking smooth extensions of $X$ and $Y$ to local vector fields in a neighborhood $U$ of $p$, we may write
\begin{align*} D_X\eta = f_1X + f_2Y \ \ \mathrm{and} \\ D_Y\eta = g_1X + g_2Y,
\end{align*}
for some smooth functions $f_1,f_2,g_1,g_2:U\rightarrow \mathbb{R}$. Differentiating the equalities above yields
\begin{align*} D_YD_X\eta = Y(f_1)X + f_1D_YX + Y(f_2)Y + f_2D_YY \ \ \mathrm{and} \\
D_XD_Y\eta = X(g_1)X + g_1D_XX + X(g_2)Y + g_2D_XY.
\end{align*}
Now notice that since $h(X,Y) = 0$, it follows that, at $p$, $D_XY$ and $D_YX$ are tangential. Therefore, the transversal components of the two vectors above are given by $f_2h(Y,Y)\eta$ and $g_1h(X,X)\eta$, respectively. On the other hand, $D$ is a flat connection, and hence
\begin{align*} D_XD_Y\eta - D_YD_X\eta - D_{[X,Y]}\eta = 0.
\end{align*}
Since $D_{[X,Y]}\eta$ is tangential, it follows that $f_2h(Y,Y)\eta = g_1h(X,X)\eta$, and then $f_2 = g_1$ at $p$. This shows that the matrix of $d\eta_p$ written in the basis $\{X,Y\}$ is symmetric, and we have what we wished to prove.

\end{proof}

The rank of the affine fundamental form does not depend on the transversal vector field, as it becomes clear in \cite[Proposition 2.5]{nomizu}. This opens the possibility of working with an auxiliary Euclidean structure in $\mathbb{R}^3$. Let $\langle \cdot,\cdot \rangle$ denote the usual inner product, and denote by $||\cdot||_e$ the induced Euclidean metric. Let $B_e$ and $\partial B_e$ be the Euclidean unit ball and sphere, respectively. Denote by $u:\partial B_e \rightarrow \partial B$ a (smooth) map which carries each vector $v \in \partial B_e$ to a respective vector $u(v) \in \partial B$ at which the supporting hyperplane of $\partial B$ at $u(v)$ is the same supporting hyperplane of $\partial B_e$ at $v$. Notice that there are exactly two choices of such a smooth map $u$, and that this map can be regarded as the Birkhoff-Gauss map of the Euclidean circle in the geometry given by $||\cdot||$, or as the inverse of the Euclidean Gauss map of $\partial B$ as an immersed surface.\\

Writing $d\xi_pX = du^{-1}_{\eta(p)}\circ d\eta_pX$ and recalling that $du^{-1}_{\eta(p)}$ is self-adjoint (because it is the usual Gauss map of the Minkowski unit sphere), we have the following useful expressions for $h$:
\begin{align} \label{exph} h(X,Y) =-\frac{\langle Y,d\xi_pX\rangle}{\langle\eta,\xi \rangle}= -\frac{\langle du^{-1}_{\eta(p)}Y,d\eta_pX\rangle}{\langle \eta,\xi\rangle}.
\end{align}
  These equalities come from evaluating the usual inner product of both sides of the Gauss formula (\ref{gauss}) with $\xi$. Notice that it follows from the self-adjointness of $du^{-1}_{\eta(p)}$ that $d\eta_p$ is self-adjoint with respect to $h$. In other words, we have $h(X,d\eta_pY) = h(d\eta_pX,Y)$ for any $X,Y \in T_pM$. In the ``language'' of affine differential geometry this is merely the Ricci equation for an equiaffine immersion (cf. \cite[Theorem 2.4]{nomizu}). \\

Notice also that $d\eta_p = du_{\xi(p)}\circ d\xi_p$. From the strict convexity of $B$ we have that the usual Gaussian curvature of $\partial B$ is non-negative, but may be zero at some isolated point. This may create some ``artificial" direction in the kernel of $d\eta_p$, in the sense that it is not in the kernel of $d\xi_p$. In other words, the Birkhoff-Gauss map can have a rank different to that of the Euclidean Gauss map. For that reason, we will restrict ourselves to the case where $du_v$ is an isomorphism for each $v \in \partial B_e$, which is exactly the same as saying that $\partial B$ has positive Euclidean Gaussian curvature at every point. Norms that give birth to unit spheres with such a property will be called \emph{admissible norms}. In what follows, \textbf{all norms are assumed to be admissible}.

\begin{prop}\label{rank2ind} If the rank of the affine fundamental form $h$ at $p \in M$ equals $2$, and $h$ is indefinite, then we also have that the tangent space $T_pM$ admits a basis of eigenvectors of the differential map $d\eta_p$.
\end{prop}
\begin{proof} The strategy is to use the map $u:\partial B_e\rightarrow \partial B$ of the unit sphere. If $\xi$ denotes the Euclidean Gauss map of $M$, then we have $\eta = u\circ\xi$. Therefore, $d\eta_p = du_{\xi(p)}\circ d\xi_p$. It follows that
\begin{align*} \mathrm{det}(d\eta_p) = \mathrm{det}\left(du_{\xi(p)}\right)\cdot\mathrm{det}(d\xi_p).
\end{align*}
Notice that $\mathrm{det}\left(du_{\xi(p)}\right) > 0$. Indeed, $u$ is the inverse of the usual Gauss map of the strictly convex surface $\partial B$. On the other hand, if $h$ has rank $2$ and is not definite, then we may choose vectors $X,Y \in T_pM$ such that $h(X,X) = -h(Y,Y) \neq 0$ and $h(X,Y) = 0$. From (\ref{exph}) it follows that $\langle X,d\xi_pX\rangle$ and $\langle Y,d\xi_pY\rangle$ have different signs, and $\langle X,d\xi_pY\rangle = \langle d\xi_pX,Y\rangle = 0$. Therefore, we have that $\mathrm{det}(d\xi_p) < 0$. Hence $\mathrm{det}(d\eta_p) < 0$, and this finishes the proof. 

\end{proof}

The cases where the rank of $h$ is $0$ or $1$ are discussed next. These cases are somehow special because, as in the Euclidean subcase (where $h$ is the usual second fundamental form), when they occur we can guarantee the existence of a null eigenvalue for $d\eta_p$.

\begin{prop}\label{rank01} Let $f:M\rightarrow (\mathbb{R}^3,||\cdot||)$ be an immersed surface, with Birkhoff normal vector field $\eta$ and induced affine fundamental form $h$. If at $p \in M$ the rank of $h$ is $0$, then $d\eta_p = 0$. In the case that the rank of $h$ at $p\in M$ equals $1$, we have that $d\eta_p$ has two distinct eigenvectors, and one of them is associated to the eigenvalue $0$.
\end{prop}
\begin{proof} On the surface $M$, let $\eta$ denote the Birkhoff-Gauss map and $\xi$ denote the usual Euclidean Gauss map. Then we have $\eta = u \circ \xi$. Let $X,Y \in T_pM$, and denote
 by these letters also respective extensions of these vectors to local vectors fields. The Gauss equation for the Euclidean normal map reads
\begin{align*} D_XY = \overline{\nabla}_XY + \overline{h}(X,Y)\xi,
\end{align*}
and hence we may write
\begin{align*} \overline{h}(X,Y) = \langle D_XY,\xi\rangle = -\langle Y,D_X\xi\rangle.
\end{align*}
If, at a point $p \in M$, the bilinear form $h$ has rank $0$, then so has $\overline{h}$. Therefore, in this case we have $\langle Y,d\xi_p(X)\rangle = 0$ for any $X,Y \in T_pM$. It follows that $d\xi_p = 0$. Since $\eta = u\circ\xi$, we have $d\eta_p = du_{\xi(p)}\circ d\xi_p$, and then $d\eta_p = 0$. \\

If the rank of $h$ equals $1$, then let $X$ be the null direction of $\overline{h}$. We have that $\langle Z,D_X\xi\rangle = 0$ for any $Z \in T_pM$, and hence $d\xi_p(X) = 0$. It follows immediately that $d\eta_p(X) = 0$. Since $d\eta_p$ is not null and $du_{\xi(p)}$ is invertible (recall that the norm is admissible), the existence of another eigenvector comes from standard linear algebra arguments.

\end{proof}

\begin{remark} The auxiliary Euclidean structure could have been used to prove also Proposition \ref{rank2}. However, the different method used there emphasizes that in this case the admissibility hypothesis for the norm is not necessary.
\end{remark}

\begin{definition} Let $f:M\rightarrow (\mathbb{R}^3,||\cdot||)$ be a surface immersion with Birkhoff normal vector field $\eta$. Let $\lambda_1,\lambda_2 \in \mathbb{R}$ be the eigenvectors of $d\eta_p$. Then these numbers are called the \emph{principal curvatures} of $M$ at $p$, and the respective eigenvectors $E_1,E_2 \in T_pM$ are called the \emph{principal directions} of $M$ at $p$. The numbers
\begin{align*} K = \lambda_1\lambda_2 \ \ \ \mathrm{and} \ \ \ H = \frac{\lambda_1+\lambda_2}{2}
\end{align*}
are called the \emph{Minkowski Gaussian curvature} and the \emph{Minkowski mean curvature} of $M$ at the point $p$, respectively.
\end{definition}

These definitions are immediate extensions of the Euclidean versions, and we will discuss some of their properties in order to understand them in a better way. First of all, it is easy to see that if we set $\lambda_1 \geq \lambda_2$, then the maps $\lambda_1,\lambda_2:M\rightarrow\mathbb{R}$ are smooth in all points, possibly except for the umbilic ones (since in these points $\lambda_1$ and $\lambda_2$ may ``change their roles''). Regarding the principal directions, it is known that in the Euclidean subcase they must be orthogonal. This is not the case in our context, but we can provide some information about them in terms of the affine fundamental form.

\begin{lemma}\label{princconj} Let $p \in M$ be a point whose principal curvatures are different, and let $V_1,V_2 \in T_pM$ be the principal directions. Then we have $h(V_1,V_2) = 0$.
\end{lemma}
\begin{proof} Let $\lambda_1(p),\lambda_2(p) \in \mathbb{R}$ be the principal curvatures of $M$ at $p$, and assume that $\lambda_1(p) > \lambda_2(p)$. By continuity, this inequality is true in a small neighborhood $U$ containing $p$, and then we may consider local vector fields $V_1,V_2\in C^{\infty}(TU)$ such that $V_1$ and $V_2$ are principal directions associated to the respective principal curvatures at each point $p \in U$. Therefore, we have $D_{V_1}\eta = \lambda_1V_1$ and $D_{V_2}\eta = \lambda_2V_2$ in $U$. Differentiating these expressions, we get
\begin{align*} D_{V_2}D_{V_1}\eta = V_2(\lambda_1)V_1 + \lambda_1D_{V_2}V_1 = V_2(\lambda_1) + \lambda_1\nabla_{V_2}V_1 + \lambda_1h(V_2,V_1)\eta \ \ \mathrm{and}\\
D_{V_1}D_{V_2}\eta = V_1(\lambda_2)V_2 + \lambda_2D_{V_1}V_2 = V_1(\lambda_2)V_2 + \lambda_2\nabla_{V_1}V_2 + \lambda_2h(V_1,V_2)\eta.
\end{align*}
Now, since $0 = D_{V_1}D_{V_2}\eta - D_{V_2}D_{V_1}\eta - D_{[V_1,V_2]}\eta$, and the vectors $V_1,V_2,\nabla_{V_1}V_2,\nabla_{V_2}V_1$ and $D_{[V_1,V_2]}\eta$ are all tangential, it follows that $(\lambda_1-\lambda_2)h(V_1,V_2)\eta = 0$. Since the principal curvatures at $p$ are distinct, the desired follows.

\end{proof}

In Euclidean differential geometry, we say that two directions given by non-zero vectors $X,Y \in T_pM$ are \emph{conjugate} if $\langle X,d\xi_pY\rangle = 0$ (or, equivalently, $\langle Y,d\xi_pX\rangle = 0$), where we recall that $\xi$ is the usual Euclidean Gauss map. However, this definition does not depend on the inner product. Indeed, it is easy to see that $X$ and $Y$ are conjugate if and only if $D_XY$ is tangential. Therefore, we may extend this definition to the Minkowski context, and the conjugate directions in the Minkowski norm will be precisely the same as the conjugate directions in the Euclidean norm. In particular, from the previous lemma it follows that the principal directions at any point are always conjugate directions. We summarize all this as follows.

\begin{lemma} Let $f:M\rightarrow(\mathbb{R}^3,||\cdot||)$ be an immersed surface with Birkhoff-Gauss map $\eta$ and usual Euclidean Gauss map $\xi$. For any non-zero vectors $X,Y \in T_pM$, the following statements are equivalent:\\

\noindent\emph{\textbf{(a)}} the vectors $X$ and $Y$ are conjugate directions in the Euclidean sense, \\

\noindent\emph{\textbf{(b)}} the derivative $D_XY$ is tangential, and  \\

\noindent\emph{\textbf{(c)}} $h(X,Y) = 0$. \\
\end{lemma}
\begin{proof} From the equality $D_XY = \nabla_XY + h(X,Y)\eta$ we have that $D_XY$ is tangential if and only if $h(X,Y) = 0$. Now recall that from (\ref{exph}) we have
\begin{align*} h(X,Y) = \frac{\langle D_XY,\xi\rangle}{\langle \eta,\xi \rangle} = -\frac{\langle Y,D_X\xi\rangle}{\langle \eta,\xi\rangle}.
\end{align*}
Since the derivative $D_X\xi$ at a point $p \in M$ is precisely $d\xi_p(X)$, the proof is complete.

\end{proof}

Still in this direction, recall that positive Gaussian curvature has a sort of geometric consequence that can be regarded independently of the norm. Namely, if $M$ has positive Gaussian curvature at $p \in M$, then the normal vector to any curve on $M$ points at $p$ to the side of the tangent plane $T_pM$. This inspires the next proposition, which is important in Section 3 of \cite{diffgeom3}, for proving analogues to Hadamard theorems.

\begin{prop}\label{gaussposit} Assume, as usual, that $||\cdot||$ is an admissible norm, and let $f:M\rightarrow (\mathbb{R}^3,||\cdot||)$ be an immersed surface. Denote by $K$ and $K_e$ the Minkowski and the Euclidean Gaussian curvatures of $M$, respectively. For a point $p \in M$, the following statements are equivalent:\\

\noindent\emph{\textbf{(a)}} $K_e(p) > 0$,\\

\noindent\emph{\textbf{(b)}} $K(p) > 0$, and\\

\noindent\emph{\textbf{(c)}} $h$ is a (positive or negative) definite bilinear form at $p$.
\end{prop}

\begin{proof} The Euclidean second fundamental form is given by $(X,Y)\mapsto\langle Y,d\xi_pX\rangle$, and it is known that this bilinear form is definite if and only if $K_e(p) > 0$. Hence, from equality (\ref{exph}) it follows that \textbf{(a)}$\Leftrightarrow$\textbf{(c)}. \\

Assume that \textbf{(a)} holds, and let $V_1,V_2 \in T_pM$ be such that $d\eta_pV_1 = \lambda_1V_1$ and $d\eta_pV_2 = \lambda_2V_2$. We have to prove that $\lambda_1$ and $\lambda_2$ have the same sign. To do so, we notice that $d\xi_pV_1 = \lambda_1du_{\eta(p)}^{-1}V_1$, and the analogous equality holds for $V_2$. Therefore,

\begin{align*} \langle d\xi_pV_1,V_1\rangle = \lambda_1\langle du^{-1}_{\eta(p)}V_1,V_1\rangle \ \ \mathrm{and}\\
\langle d\xi_pV_2,V_2 \rangle = \lambda_2\langle du^{-1}_{\eta(p)}V_2,V_2\rangle.
\end{align*}
Since $K_e(p) > 0$, it follows that the left hand terms have the same sign. The same holds for the numbers $\langle du^{-1}_{\eta(p)}V_1,V_1 \rangle$ and $\langle du^{-1}_{\eta(p)}V_2,V_2\rangle$, due to the admissibility of the norm. Then $\lambda_1$ and $\lambda_2$ have the same sign, and this implies $K(p) > 0$. To prove \textbf{(b)}$\Rightarrow$\textbf{(a)}, the same argument works. Notice also that, in this case, the question whether $h$ is positive or negative only depends on the orientation chosen to $\eta$.

\end{proof}

Extending another concept from the Euclidean subcase, we say that a non-zero vector $X\in T_pM$ is an \emph{asymptotic direction} if $X$ is conjugate to itself. In other words, we say that $X \in T_pM\setminus\{0\}$ is an asymptotic direction whenever $D_XX \in T_pM$. It turns out that, in view of this definition, the asymptotic directions of a surface with respect to the Minkowski norm are precisely the same as the ones with respect to the Euclidean norm. \\

A connected regular curve $\gamma:J\rightarrow M$ is said to be an \emph{asymptotic curve} if the tangent line of $\gamma$ at any point is an asymptotic direction. It is clear that if the norm is admissible, then the asymptotic lines with respect to the norm are precisely the same as the ones with respect to Euclidean geometry. In Euclidean differential geometry, the asymptotic directions are characterized as the ones for which $\langle X,d\xi_pX\rangle = 0$. In the general case, when the space is endowed with a Minkowski norm, we clearly can characterize the asymptotic directions as the directions $X \in T_pM$ for which $h(X,X) = 0$. Therefore, as in the Euclidean subcase, if there exists an asymptotic direction at a point $p \in M$ where the eigenvalues of $d\eta_p$ are both non-zero, then we may guarantee the existence of a local asymptotic curve passing through $p$ via standard ordinary differential equations theory.  \\

Notice that the principal curvatures of a plane are $0$ at any point. Also, any Minkowski sphere has constant, equal principal curvatures at each of its points. Indeed, the Birkhoff-Gauss map can be regarded as the map $\eta:M\rightarrow \partial B$ given by $\eta(x)=\frac{1}{\rho}(x-p)$, where $p$ is the center $M$ and $\rho$ is the radius. Clearly, $d\eta_p = \frac{1}{\rho}\mathrm{Id}_{T_{p}M}$, and hence the principal curvatures of $M$ at any point $p$ are $\frac{1}{\rho}$.   \\

We say that a point $p \in M$ is an \emph{umbilic point} if the principal curvatures of $M$ at $p$ have the same value. Equivalently, a point $p \in M$ is umbilic when the differential of the Birkhoff normal vector field at $p$ is a multiple of the identity map. By the previous observation, all the points of a plane or of a Minkowski sphere are umbilic. The next proposition states that, as in the Euclidean subcase, these are the only possible surfaces with such property.

\begin{prop}\label{allumbilic} An immersed connected surface all whose points are umbilic is contained in a plane or in a Minkowski sphere.
\end{prop}
\begin{proof} For each $p \in M$ we have that $d\eta_p(X) = \lambda(p)X$ for any $X \in T_pM$, where $\lambda:M\rightarrow\mathbb{R}$ is a smooth function. Our first step is to prove that the function $\lambda$ is constant. For this sake, fix linearly independent vectors $X,Y \in T_pM$ and denote also by $X$ and $Y$ the parallel transport of $X$ through a curve tangent to $Y$ at $p$, and the parallel transport of $Y$ through a curve tangent to $X$ at $p$, both with respect to the induced connection $\nabla$. We have then $\nabla_XY|_p = \nabla_YX|_p = 0$. Now, extending smoothly both vector fields to an open neighborhood of $p$, we may calculate at $p$
\begin{align*} D_YD_X\eta = D_Y(\lambda X) = Y(\lambda)X + \lambda h(Y,X)\eta \ \ \mathrm{and} \\
D_XD_Y\eta = X(\lambda)Y + \lambda h(X,Y)\eta.
\end{align*}
Since $D$ is a flat connection, we write
\begin{align*} 0 = D_YD_X\eta - D_XD_Y\eta - D_{[X,Y]}\eta = Y(\lambda)X-X(\lambda)Y - \lambda[X,Y].
\end{align*}
Recalling that $\nabla$ is a torsion-free connection, it follows that $[X,Y] = 0$ at $p$. Hence we have $X(\lambda)Y - Y(\lambda)X = 0$, and this gives $X(\lambda) = Y(\lambda) = 0$ (since $X$ and $Y$ are linearly independent). This argument shows that the derivative of the function $\lambda$ at any point $p \in M$ and with respect to any direction $X \in T_pM$ equals $0$. It follows that $\lambda$ is constant. \\

If $\lambda = 0$, then the Birkhoff-Gauss map is constant, and this means that the Birkhoff normal vector is the same for each point of $M$. In particular, the Euclidean normal vector is also the same for every point, and therefore $M$ is contained in a plane (see \cite{manfredo}). If $\lambda \neq 0$, then the map $x \in M \mapsto x - \frac{1}{\lambda}\eta(x) \in \mathbb{R}^3$ is clearly a constant map. Indeed, for any point $p \in M$ and any direction $X \in T_pM$, we have
\begin{align*} D_X\left(x - \frac{1}{\lambda}\eta(x)\right) = X - \frac{1}{\lambda}(\lambda X) = 0.
\end{align*}
Thus, $M$ is contained in the Minkowski sphere whose center is this constant point, and whose radius equals $\frac{1}{\lambda}$.

\end{proof}

A regular connected curve $\gamma:J\subseteq\mathbb{R}\rightarrow M$ is said to be a \emph{curvature line} if for each $t \in J$ the tangent vector $\gamma'(t)$ gives a principal direction at $\gamma(t)$. We will characterize the curvature lines of a surface in a Minkowski space in a similar manner as it is done for the Euclidean subcase.

\begin{prop} Let $\gamma:J\rightarrow M$ be a regular connected curve. Then $\gamma$ is a curvature line of $M$ if and only if there exists a function $\lambda:J\rightarrow\mathbb{R}$ such that
\begin{align*} (\eta\circ\gamma)'(t) = \lambda(t)\gamma'(t),
\end{align*}
for each $t \in J$.
\end{prop}
\begin{proof} First suppose that $\gamma$ is a curvature line. Then, for each $t \in J$, we have that $\gamma'(t)$ is a principal direction, and therefore $(\eta\circ\gamma)'(t) = d\eta_{\gamma(t)}(\gamma'(t)) = \lambda(t) \gamma'(t)$, where $\lambda(t)$ is an eigenvalue of $d\eta_{\gamma(t)}$. \\

Conversely, assume that $\gamma$ is a connected curve for which $(\eta\circ\gamma)'(t) = \lambda(t)\gamma'(t)$ holds for some function $\lambda:J\rightarrow\mathbb{R}$. For each $t \in J$ we have that $\gamma'(t)$ is an eigenvector of $d\eta_{\gamma(t)}$. Thus, $\gamma$ is a curvature line.

\end{proof}

\section{An analogue of the normal curvature}\label{normal}

Throughout this section \textbf{we still always assume that the norm fixed in the space is admissible}. As usual, we let $u:\partial B_e\rightarrow \partial B$ be the inverse of the Euclidean Gauss map of $\partial B$. Recall also that we are denoting by $\langle\cdot,\cdot\rangle$ the \emph{usual inner product in} $\mathbb{R}^3$. Given an immersed surface $f:M\rightarrow\mathbb{R}^3$, we still denote by $\eta$ and $\xi$ the Birkhoff-Gauss and usual Euclidean Gauss maps of $M$, respectively. \\

In Euclidean differential geometry, the \emph{normal curvature} of a surface $M$ in a given point $p \in M$ and a given direction $X \in T_pM$ can be regarded as the (signed) length of the projection of the normal vector of a curve in $M$, passing through $p$ with tangent vector $X$, onto $\xi(p)$. In particular, the considered curve can be taken as the intersection of the plane spanned by $\xi(p)$ and $X$, and therefore the normal curvature is the usual curvature of this (plane) curve at $p$ (see \cite{manfredo}). This observation allows us to extend this notion to our general case. Let $f:M\rightarrow(\mathbb{R}^3,||\cdot||)$ be an immersed surface, and fix $p \in M$ and $X \in S_p\subseteq T_pM$, where $S_p$ denotes the unit circle of $T_pM$. Denote by $H$ the plane spanned by $\eta(p)$ and $X$. Let $\gamma:(-\varepsilon,\varepsilon)\rightarrow M$ be a local arc-length parametrization of the curve given by the intersection of the plane $p\oplus H$ with $M$, and assume that $\gamma(0) = p$ and $\gamma'(0) = X$.

\begin{definition}The \emph{Minkowski normal curvature} of $M$ at $p \in M$ in the direction $X \in S_p$ is the circular curvature of $\gamma$ at $p$ in the plane geometry endowed in $H$ by the norm $||\cdot||$ (in other words, the geometry in $H$ whose unit circle is the intersection of $\partial B$ with $H$). We will denote this number by $k_{M,p}(X)$.
\end{definition}

We will give a formula for the Minkowski normal curvature in terms of the auxiliary Euclidean structure fixed in the plane. To do so, we first notice that this is essentially a problem in the plane $H$. Following \cite{Ba-Ma-Sho}, the circular curvature of $\gamma$ at $p$ is the ratio between its usual plane Euclidean curvature and the usual plane Euclidean curvature of the circle $\partial B\cap H$ at a point whose tangent lies in the direction $X$.

\begin{teo} For any $p \in M$ and $X \in T_pM$ we have
\begin{align}\label{normalcurv} k_{M,p}(X) = \frac{\langle du^{-1}_{\eta(p)}X,d\eta_pX\rangle}{\langle du^{-1}_{\eta(p)}X,X\rangle},
\end{align}
where we are considering the natural identification $T_pM \simeq T_{\eta(p)}\partial B \simeq T_{\xi(p)}\partial B_e$.
\end{teo}

\begin{proof} Let us first look at $\partial B$ as an immersed surface. The Euclidean normal curvature of $\partial B$ at $\eta(p)$ in the direction $X$ is given by $\langle du^{-1}_{\eta(p)}X,X\rangle$, since $u^{-1}$ is the Euclidean Gauss map of $\partial B$. Following \cite{manfredo}, this normal curvature can be obtained from the curve $\varphi:=\partial B\cap H$ as
\begin{align}\label{eq1} -\langle du^{-1}_{\eta(p)}X,X\rangle = k_{\varphi}(\eta(p))\langle \zeta,\xi(p)\rangle,
\end{align}
where $k_{\varphi}(\eta(p))$ is the (plane) Euclidean curvature of the curve $\varphi$ at $\eta(p)$, and $\zeta$ is the unit Euclidean normal vector to $X$ at the plane $H$, which is also the Euclidean normal vector of the curve $\varphi$ at $\eta(p)$. On the other hand, the Euclidean normal curvature of $M$ at $p$ in the direction $X$ is given by $\langle d\xi_pX,X\rangle$, since $\xi$ is the Euclidean Gauss map of $M$. As in the previous argument, this normal curvature can be obtained from the curve $\gamma$ as
\begin{align}\label{eq2} -\langle d\xi_pX,X\rangle = k_{\gamma}(p)\langle \zeta,\xi(p)\rangle,
\end{align}
where $k_{\gamma}(p)$ is the (plane) Euclidean curvature of $\gamma$ at $p$. Now, from (\ref{eq1}) and (\ref{eq2}) we have

\begin{align*} k_{M,p}(X) = \frac{k_{\gamma}(p)}{k_{\varphi}(\eta(p))} = \frac{\langle d\xi_pX,X\rangle}{\langle du^{-1}_{\eta(p)}X,X\rangle}.
\end{align*}
Since $\xi= u^{-1}\circ\eta$, and since the differential of the Euclidean Gauss map of any immersed surface is self-adjoint at any point, it follows that $\langle d\xi_pX,X\rangle = \langle du^{-1}_{\eta(p)}X,d\eta_pX\rangle$. This gives equality (\ref{normalcurv}).

\end{proof}

We will derive three consequences of this formula (see the next three corollaries). First we prove that, as in the Euclidean subcase, if the normal curvature of a surface is constant, then this surface must be a plane or a Minkowski circle. After that, we will find a relation between the principal directions and the normal curvature. Finally, we show that the asymptotic directions at a point can be characterized in terms of the normal curvature.

\begin{coro}\label{normalumb} The Minkowski normal curvature of an immersed connected surface is constant if and only if this surface is contained in a plane or in a Minkowski sphere. The first case occurs if and only if $k_{M,p} = 0$, and in the second case the radius of the sphere is given by $\left|k_{M,p}\right|^{-1}$.
\end{coro}
\begin{proof} First, it is clear that the Minkowski normal curvature of a plane is always zero, since each plane section yields a straight line segment. Also, if $M$ is a Minkowski sphere of radius $\lambda\in\mathbb{R}$, then we may assume for simplicity that it is centered at the origin, and hence the Birkhoff-Gauss map can be regarded as $\eta(p) = \frac{1}{\lambda}p$. Therefore, $d\eta_p(X) = \frac{1}{\lambda}X$, and we have $k_{M,p}(X) = \frac{1}{\lambda}$.  \\

Assume now that the Minkowski normal curvature of a surface $M$ equals $\lambda \in \mathbb{R}$. Then we have that
\begin{align}\label{eq3} \langle du^{-1}_{\eta(p)}X,d\eta_pX\rangle = \lambda\langle du^{-1}_{\eta(p)}X,X\rangle,
\end{align}
for any $p \in M$ and non-zero $X \in T_pM$. Recalling that $du^{-1}_{\eta(p)}$ is a self-adjoint operator, we may take orthogonal unit (in the Euclidean norm) vectors $E_1,E_2 \in T_pM$ such that
\begin{align*} du^{-1}_{\eta(p)}E_1 = \mu_1E_1 \ \ \mathrm{and} \\
du^{-1}_{\eta(p)}E_2 = \mu_2E_2,
\end{align*}
for some $\mu_1,\mu_2 \in \mathbb{R}$. Moreover, since the norm is admissible we may assume that $\mu_1,\mu_2 > 0$. From (\ref{eq3}), it follows that
\begin{align*} d\eta_pE_1 = \lambda E_1 + \alpha E_2 \ \ \mathrm{and} \\
d\eta_p E_2 = \lambda E_2 + \beta E_1,
\end{align*}
for some numbers $\alpha,\beta \in \mathbb{R}$. Applying the equality (\ref{eq3}) to the vector $E_1 + E_2$, we get
\begin{align}\label{eq4} \mu_1\beta + \mu_2\alpha = 0.
\end{align}
In order to obtain another relation between these numbers, we recall that $\xi$ is the Euclidean Gauss map of $M$, and hence $d\xi_p$ is self-adjoint for any $p \in M$. We may write then $\langle d\xi_pE_1,E_2\rangle = \langle E_1,d\xi_pE_2\rangle$. But since $\xi = u\circ \eta$, this equality reads
\begin{align*} \left\langle \frac{\lambda}{\mu_1}E_1+\frac{\alpha}{\mu_2}E_2,E_2\right\rangle =  \left\langle E_1,\frac{\lambda}{\mu_2}E_2 + \frac{\beta}{\mu_1}E_1\right\rangle.
\end{align*}
Therefore we have $\alpha\mu_1 - \beta\mu_2 = 0$. This equality, together with (\ref{eq4}) and the fact that $\mu_1,\mu_2 \neq 0$, gives that $\alpha = \beta = 0$. This means that, for each $p \in M$, we have $d\eta_p = \lambda \mathrm{Id}|_{T_pM}$. The result comes now from Proposition \ref{allumbilic}.

\end{proof}

\begin{remark} We do not need an auxiliary Euclidean structure to prove that any Minkowski sphere has constant Minkowski normal curvature. Indeed, any normal section $H$ on a Minkowski sphere will yield a curve which is a homothet of the unit circle of $(H,||\cdot||)$. Hence, its circular curvature will be the inverse of the radius.
\end{remark}

\begin{coro}\label{princnorm} Let $V \in T_pM$ be a principal direction at $p \in M$, with associated principal curvature $\lambda$. Then $k_{M,P}(V) = \lambda$.
\end{coro}
\begin{proof} Since $d\eta_p(V) = \lambda V$, it follows that
\begin{align*} k_{M,p}(V) = \frac{\langle du^{-1}_{\eta(p)}V,\lambda V\rangle}{\langle du^{-1}_{\eta(p)}V,V\rangle} = \lambda.
\end{align*}

\end{proof}

\begin{coro}\label{asympnorm} A non-zero vector $X \in T_pM$ is an asymptotic direction if and only if $k_{M,p}(X) = 0$.
\end{coro}
\begin{proof} From (\ref{exph}) and (\ref{normalcurv}) we have the equality
\begin{align*} k_{M,p}(X) = -\frac{h(X,X)\langle \eta,\xi \rangle}{\langle du^{-1}_{\eta(p)}X,X\rangle}.
\end{align*}
Therefore, $k_{M,p}(X) = 0$ if and only if $h(X,X) = 0$.

\end{proof}

Again, here the general Minkowski case presents a similar behavior as the Euclidean one. Also, the umbilic points can be characterized as the points where the normal curvature is the same for every direction. We will prove this now.

\begin{prop}\label{normumb} A point $p \in M$ is umbilic if and only if $k_{M,p}$ is constant in $T_pM\setminus\{0\}$. In this case, $k_{M,p}$ equals the principal curvature of $M$ at $p$.
\end{prop}
\begin{proof} Assume first that $p \in M$ is umbilic, and let $\lambda$ be the value of the principal curvature of $M$ at $p$. If $X \in T_pM$, then
\begin{align*} k_{M,p}(X) = \frac{\langle du^{-1}_{\eta(p)}X,\lambda X\rangle}{\langle du^{-1}_{\eta(p)}X,X\rangle} = \lambda.
\end{align*}
Now suppose that $k_{M,p} = \lambda$ is constant. Then we have
\begin{align*} \langle du^{-1}_{\eta(p)}X,d\eta_pX\rangle = \lambda\langle du^{-1}_{\eta(p)}X,X\rangle
\end{align*}
for every $X \in T_pM$. From the same argument as used in Corollary \ref{normalumb} it follows that $d\eta_p = \lambda\mathrm{Id}|_{T_{p}M}$. Hence $p$ is an umbilic point with principal curvature $\lambda$.

\end{proof}

In the Euclidean subcase, the principal curvatures of $M$ at a point $p$ are precisely the maximum and the minimum values of the normal curvature in this point. It might be a little surprising that this is also true in the general Minkowski case. Namely, we have

\begin{teo} Let $\lambda_1\geq\lambda_2$ be the principal curvatures of $M$ at $p$. Then we have the inequalities $\lambda_2\leq k_{M,p}(V)\leq \lambda_1$, for every $V \in T_pM\setminus\{0\}$. Moreover, equality holds only in the respective principal directions.
\end{teo}
\begin{proof} First notice that if both principal curvatures are equal, then the result comes straightforwardly (Proposition \ref{normumb}). Let $V_1,V_2 \in T_pM$ be the principal directions associated to $\lambda_1$ and $\lambda_2$, respectively, and asssume that the principal curvatures are distinct. Recall that $V_1$ and $V_2$ are conjugate directions, which means that $h(V_1,V_2) = 0$ (Lemma \ref{princconj}). From (\ref{exph}) it follows that
\begin{align*} \langle du^{-1}_{\eta(p)}V_1,d\eta_pV_2\rangle = \langle du^{-1}_{\eta(p)}V_2,d\eta_pV_1\rangle = 0.
\end{align*}
Hence we have
\begin{align*} \langle du^{-1}_{\eta(p)}V_1,V_2\rangle = \frac{1}{\lambda_2}\langle du^{-1}_{\eta(p)}V_1,d\eta_pV_2\rangle = 0,
\end{align*}
and the same argument holds for $\langle du^{-1}_{\eta(p)}V_2,V_1\rangle$. If $V_2$ is, say, associated to the eigenvalue zero, then
\begin{align*} \langle du^{-1}_{\eta(p)}V_1,V_2\rangle = \langle V_1,du^{-1}_{\eta(p)}V_2\rangle = \frac{1}{\lambda_1}\langle d\eta_pV_1,du^{-1}_{\eta(p)}V_2\rangle = 0.
\end{align*}

\noindent Summarizing this, for $\lambda_1\neq\lambda_2$ we necessarily have
\begin{align*}\langle du^{-1}_{\eta(p)}V_1,V_2\rangle = \langle du^{-1}_{\eta(p)}V_2,V_1\rangle = 0.
\end{align*}
Therefore, decomposing a non-zero vector $V \in T_pM$ by $V = \alpha V_1 + \beta V_2$ yields the equality
\begin{align*} k_{M,p}(V) = \frac{\langle du^{-1}_{\eta(p)}V,\alpha\lambda_1V_1 + \beta\lambda_2V_2\rangle}{\langle du^{-1}_{\eta(p)}V,V\rangle} = \frac{\alpha^2\lambda_1\langle du^{-1}_{\eta(p)}V_1,V_1\rangle +  \beta^2\lambda_2\langle du^{-1}_{\eta(p)}V_2,V_2\rangle}{\alpha^2\langle du^{-1}_{\eta(p)}V_1,V_1\rangle + \beta^2\langle du^{-1}_{\eta(p)}V_2,V_2\rangle}.
\end{align*}
Using again the fact that $u^{-1}$ is the (Euclidean) Gauss map of $\partial B$, we may assume that the bilinear form associated to $du^{-1}_{\eta(p)}$ is positive definite. Hence the equality above yields immediately $\lambda_2 \leq k_{M,p}(V) \leq \lambda_1$. The claim on the equality cases is also immediate. \\

\end{proof}

\begin{remark} In the Euclidean case, at any point the sum of the normal curvatures in orthogonal of complementary directions is constant (cf. \cite{manfredo}). We find something similar in \cite{diffgeom2}, by endowing the surface with a certain Riemannian metric whose orthogonality relation ``organizes" the directions in the same way.  
\end{remark}

\begin{coro} Let $f:M\rightarrow(\mathbb{R}^3,||\cdot||)$ be an immersed surface, and let $G:(\mathbb{R}^3,||\cdot||)\rightarrow(\mathbb{R}^3,||\cdot||)$ be an isometry of the space. Then, up to the sign, the normal curvature of $M$ at each point is invariant under $G$. In particular, the Minkowski Gaussian curvature of $M$ is invariant under the action of $G$.
\end{coro}
\begin{proof} Let $p \in M$ and $X \in T_pM$. We have to prove that $k_{M,p}(X) = k_{G(M),G(p)}(G(X))$. Let $H$ be the plane spanned by $X$ and $\eta(p)$, and calculate the normal curvature $k_{M,p}(X)$ at $p$ as the circular curvature of the curve $\gamma$ determined by the section of $M$ through $H$.\\

 Since Birkhoff orthogonality is invariant under isometry, we have that $\eta(G(p)) = G(\eta(p))$, and hence the plane spanned by $G(X)$ and $\eta(G(p))$ is precisely $G(H)$, which is isometric to $H$ by $G$. Also, every isometry of a normed space is a composition of a linear map with a translation (Mazur-Ulam theorem, cf. \cite{thompson}). Therefore, the normal curvature $k_{G(M),G(p)}(G(X))$ is the circular curvature of the curve $G(\gamma)$ at $G(p)$. The desired follows now from the fact that the circular curvature is, up to the sign, invariant under isometries of the plane (see \cite[Theorem 7.2]{Ba-Ma-Sho}). Whether or not the signs at $p$ change depends on the restriction of $G$ to $H$, being or not being orientation preserving. In both cases the Minkowski Gaussian curvature remains the same.

\end{proof}

\begin{remark} The corollary above states that the Minkowski Gaussian curvature is invariant under an isometry of the ambient space. However, a natural extension of the concept of isometry for surfaces immersed in a Minkowski space would be as follows: let $M$, $N$ be two surfaces immersed in $(\mathbb{R}^3,||\cdot||)$. A smooth mapping $F:M\rightarrow N$ is said to be an \emph{isometry}  (in the induced norm) if $||dF_pV|| = ||V||$ for any $p \in M$ and $V \in T_pM$. It seems to be difficult to decide whether the Minkowski Gaussian curvature is invariant under isometries. 
\end{remark}

\section{Some rigidity theorems} \label{rigid}

In this section we want to prove that, with a certain additional hypothesis, any compact, connected surface immersed in a space endowed with an admissible norm which has constant positive Minkowski Gaussian curvature, or constant mean curvature, is a Minkowski circle. Our proof follows the steps of the proof given in \cite{manfredo} for the Euclidean subcase. We will start with two auxiliary lemmas.

\begin{lemma}\label{point} Let $f:M\rightarrow (\mathbb{R}^3,||\cdot||)$ be an immersion of a compact, connected surface without boundary. Then there exists a point $p \in M$ such that the product of the principal curvatures of $M$ at $p$ is positive.
\end{lemma}
\begin{proof} Let $B_M$ be the smallest (closed) Minkowski ball centered at the origin and containing $M$. Then it is clear that there exists a point $p \in B_M\cap M$, and that $T_pM = T_pB_M$. Taking any normal section, and regarding the normal curvatures of $B_M$ and $M$ the circular curvatures of the respective intersection curves, it becomes clear that the normal curvature of $B_M$ is greater than or equal to the normal curvature of $M$ (indeed, the opposite would make it possible to construct closed, convex curves contradicting \cite[Theorem 8.3]{Ba-Ma-Sho}). Since the principal curvatures of $M$ at $p$ are the normal curvatures in the associated principal directions, it follows that their product is greater than or equal to $1/r^2$, where $r$ is the radius of $B_M$.

\end{proof}

\begin{remark} An easier (but less geometric) proof of this lemma could be based on the usage of the same result for the Euclidean case (as it is proved in \cite{manfredo}) and Proposition \ref{gaussposit}.\\
\end{remark}

Before coming to the next auxiliary lemma, let us present an observation. Assume that the principal curvatures of $M$ at a point $p$ are distinct, and both non-zero. Then, due to continuity, this condition on the principal curvatures is true in an open neighborhood $U$ of $p$. Therefore we may set functions $\lambda_1:U\rightarrow \mathbb{R}$ and $\lambda_2:U\rightarrow\mathbb{R}$ to be the greatest and the smallest principal curvatures at $q \in U$, respectively, and hence we can choose distinct vector fields $V_1,V_2 \in C^{\infty}(U)$ such that, at each point $q \in U$, $V_1$ and $V_2$ are principal directions associated to $\lambda_1$ and $\lambda_2$, respectively. It follows from Lemma \ref{princconj} and the comments below it that $V_1$ and $V_2$ give (linearly independent) conjugate directions at any point of $U$. From Corollary 1 in Section 3.4 of \cite{manfredo} it follows that we may endow $U$ with a parametrization whose coordinate curves are tangent to the respective directions given by $V_1$ and $V_2$. In our language, this is the same as stating that we may re-scale these vector fields in such a way that we have $[V_1,V_2] = 0$ at each point of $U$. \\

Let $V_1$ and $V_2$ be the principal vector fields of a surface $M$. Then its trajectories are the curvature lines of $M$. We say that $M$ has \emph{coercively convex curvature lines} at $p \in M$ if the following hypothesis holds: let $p \in M$ be a point where $D_{V_1}V_2|_p = 0$, and let $\gamma$ be a local trajectory of $V_2$ through $p$. Then the projection of $D_{V_1}V_2|_{\gamma}$ onto $V_1$ has negative derivative at $p$. Intuitively, this means that the curvature line associated to $V_2$ ``forces $D_{V_1}V_2$ to the other side of the curve when one passes through $p$" (see Figure \ref{coercive}).

\begin{figure}[h]
\centering
\includegraphics{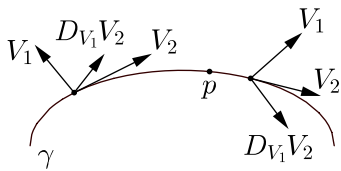}
\caption{Coercive convexity of the curvature lines of $M$ at $p$.}
\label{coercive}
\end{figure}

\begin{lemma}\label{minmax} Let $f:M\rightarrow(\mathbb{R}^3,||\cdot||)$ be an immersed surface, and let $V_1,V_2$ be the vector fields given by the principal directions associated to the principal directions $\lambda_1 \geq \lambda_2$. Let $p \in M$, where the Minkowski Gaussian curvature of $M$ is positive, and assume that $\lambda_1$ has a local maximum at $p$ and that $\lambda_2$ has a local minimum at $p$. Suppose also that $M$ has coercively convex curvature lines at $p$. Then $p$ is an umbilic point.
\end{lemma}
\begin{proof} Suppose that $p$ is not umbilic, take a neighborhood $U$ of $p$ where $\lambda_1 > \lambda_2$, and let $V_1, V_2 \in C^{\infty}(U)$ be vector fields as in the observation above. Since $[V_1,V_2]$ vanishes identically and $D$ is a flat connection, we have that $D_{V_2}D_{V_1}\eta = D_{V_1}D_{V_2}\eta$. This equality reads
\begin{align*} V_2(\lambda_1)V_1 + \lambda_1D_{V_2}V_1 = V_1(\lambda_2)V_2 + \lambda_2D_{V_1}V_2.
\end{align*}
Differentiating this equality in the directions of $V_1$ and $V_2$, respectively, we get the equalities
\begin{align*} V_1(V_2(\lambda_1))V_1 + V_2(\lambda_1)D_{V_1}V_1 + V_1(\lambda_1)D_{V_2}V_1 + \lambda_1D_{V_1}D_{V_2}V_1 = \\ =V_1(V_1(\lambda_2))V_2 + V_1(\lambda_2)D_{V_1}V_2 + V_1(\lambda_2)D_{V_1}V_2 + \lambda_2D_{V_1}D_{V_1}V_2,
\end{align*}
and
\begin{align*} V_2(V_2(\lambda_1))V_1 + V_2(\lambda_1)D_{V_2}V_1 + V_2(\lambda_1)D_{V_2}V_1 + \lambda_1D_{V_2}D_{V_2}V_1 = \\ = V_2(V_1(\lambda_2))V_2 + V_1(\lambda_2)D_{V_2}V_2 + V_2(\lambda_2)D_{V_1}V_2 + \lambda_2D_{V_2}D_{V_1}V_2.
\end{align*}

At $p$ we have $V_1(\lambda_1) = V_2(\lambda_1) = V_1(\lambda_2) = V_2(\lambda_2) = 0$, since $p$ is a local extremum of both $\lambda_1$ and $\lambda_2$. Since we have vectorial equalities, in each of them the projections of both sides onto $\eta$ must be equal. At $p$, it is clear that the only terms which have non-zero projection onto $\eta$ are the terms of the form $D_{V_i}D_{V_j}V_k$. Therefore, by the Gauss equation we have the following: looking at the restriction of the first equality to the direction $\eta$ at $p$ yields
\begin{align*} \lambda_1h(D_{V_2}V_1,V_1) = \lambda_2h(D_{V_1}V_2,V_1).
\end{align*}
Since $\lambda_1 \neq \lambda_2$ and $D_{V_1}V_2 = D_{V_2}V_1$, it follows that $h(D_{V_2}V_1,V_1) = h(D_{V_1}V_2,V_1) = 0$. Hence, at $p$, the vectors $D_{V_1}D_{V_2}V_1$ and $D_{V_1}D_{V_1}V_2$ are tangential. Repeating the argument for the second equality, we see that the same holds for $D_{V_2}D_{V_2}V_1$ and $D_{V_2}D_{V_1}V_2$. \\

Recall that from Lemma \ref{princconj} we have $h(V_1,V_2) = 0$. Evaluating (at $p$) the affine fundamental form with the vectors involved in the first equality and the vector $V_2$, we obtain
\begin{align}\label{main1} \lambda_1h(D_{V_1}D_{V_2}V_1,V_2) = V_1(V_1(\lambda_2))h(V_2,V_2) + \lambda_2h(D_{V_1}D_{V_1}V_2,V_1).
\end{align}
Doing the same with the second equality and the vector $V_1$, we get
\begin{align} \label{main2} V_2(V_2(\lambda_1))h(V_1,V_1) + \lambda_1h(D_{V_2}D_{V_2}V_1,V_1) = \lambda_2h(D_{V_2}D_{V_1}V_2,V_1).
\end{align}
We claim now that
\begin{align*}h(D_{V_1}D_{V_2}V_1,V_2) = h(D_{V_1}D_{V_1}V_2,V_2) = h(D_{V_2}D_{V_2}V_1,V_1) = h(D_{V_2}D_{V_1}V_2,V_1) = C,
\end{align*}
say. Indeed, using again the fact that $[V_1,V_2] = 0$ and recalling that $D$ is a flat connection, we have that $D_{V_2}(D_{V_1}D_{V_2}V_1) = D_{V_1}(D_{V_1}D_{V_1}V_2)$, and hence their projections onto $\eta$ are equal. Then the Gauss equation gives $h(D_{V_1}D_{V_2}V_1,V_2) = h(D_{V_1}D_{V_1}V_2,V_1)$. We use the same argument to derive the other two equalities. Now, summing up (\ref{main1}) and (\ref{main2}), we get
\begin{align}\label{main} V_2(V_2(\lambda_1))h(V_1,V_1) - V_1(V_1(\lambda_2))h(V_2,V_2) = (\lambda_2-\lambda_1)C.
\end{align}
Our last step is to prove that $C < 0$, up to re-orientation of $V_2$. For this sake, we decompose the vector field $D_{V_1}V_2$ in coordinates as $D_{V_1}V_2 = fV_1+gV_2$, for some functions $f,g:U\rightarrow\mathbb{R}$. From the previous equalities $h(D_{V_1}V_2,V_1) = h(D_{V_1}V_2,V_2) = 0$ at $p$, we have that $f(p) = g(p) = 0$. Derivating, we have
\begin{align*} D_{V_2}D_{V_1}V_2 = V_2(f) + fD_{V_2}V_1 + V_2(g)V_2 + gD_{V_2}V_2,
\end{align*}
and hence we have $C = h(D_{V_2}D_{V_1}V_2,V_1) = V_2(f)h(V_1,V_1)$ at $p$. From the coercive convexity of the curvature lines of $M$ at $p$ it follows that $V_2(f) < 0$. Finally, we look at (\ref{main}). Assume that, without loss of generality, $h$ is positive definite. Then the left hand side is non-positive, since $V_2(V_2(\lambda_1)) \leq 0$ and $V_1(V_1(\lambda_2)) \geq 0$. However, the right hand side is strictly positive, since $(\lambda_2 - \lambda_1) < 0$ and $C < 0$. This contradiction concludes the proof.

\end{proof}

We are ready now to state and prove the first theorem of this section.

\begin{teo} Let $f:M\rightarrow(\mathbb{R}^3,||\cdot||)$ be a compact, connected immersed surface without boundary, and assume that the norm $||\cdot||$ is admissible. Assume also that $M$ has coercively convex curvature lines at any point where the greatest principal curvature function attains a global maximum. If the Minkowski Gaussian curvature of $M$ is constant, then $M$ is a Minkowski sphere.
\end{teo}
\begin{proof} From Lemma \ref{point} it follows that the Minkowski Gaussian curvature of $M$ is positive everywhere. Let $\lambda_1,\lambda_2:M\rightarrow\mathbb{R}$ be the principal curvature functions of $M$, with $\lambda_1 \geq \lambda_2$. By compactness, there exists a point $p \in M$ such that $\lambda_1$ attains its maximum over $M$. Since the product $\lambda_1\lambda_2$ is constant, it follows that $\lambda_2$ attains a minimum at $p$. Hence, from Lemma \ref{minmax} we have that $\lambda_1(p) = \lambda_2(p)$. Therefore, for any $q \in M$ we have
\begin{align*} \lambda_2(p) \leq \lambda_2(q) \leq \lambda_1(q) \leq \lambda_1(p) = \lambda_2(p),
\end{align*}
and this shows that every point of $M$ is an umbilic point. From Proposition \ref{allumbilic} it follows that $M$ is contained in a Minkowski sphere. Since $M$ is connected, compact and without boundary, it follows that $M$ is both open and closed in this sphere, and therefore $M$ is the entire Minkowski sphere.

\end{proof}

Under the same hypothesis, we can also prove that a compact, connected immersed surface with positive Minkowski Gaussian curvature and constant mean curvature is a Minkowski sphere.

\begin{teo} Let $f:M\rightarrow (\mathbb{R}^3,||\cdot||)$ be a compact, connected surface immersed in an admissible Minkowski space. Assume also that $M$ has coercively convex curvature lines at any point where the greatest principal curvature function attains a global maximum. If the Minkowski Gaussian curvature is positive and the mean curvature is constant, then $M$ is a Minkowski sphere.
\end{teo}
\begin{proof} Let $p \in M$ be a point where the greatest principal curvature function $\lambda_1$ attains its maximum. Since the mean curvature is constant, it follows that $M$ attains its minimum at $p$. From Lemma \ref{minmax} it follows that $p$ is an umbilic point. As in the proof of the last theorem, it follows that every point of $M$ is umbilic, and then $M$ is a Minkowski sphere.

\end{proof}

\bibliography{bibliography.bib}

\end{document}